\documentclass[a4paper,12pt,reqno]{amsart}
\usepackage{amscd}
\usepackage{amsmath}
\usepackage{ulem}
\usepackage{amssymb}
\usepackage{amsthm}
\usepackage[usenames,dvipsnames]{color}
\usepackage{color,soul}
\usepackage{enumerate, mdwlist}  
\usepackage{mathtools}
\usepackage{tikz}
\usepackage{url}

\input xy
\xyoption{all}  

\title{ Subrings of $\Z[t]/(t^4)$}

\begin{document}

\newtheorem{thm}{Theorem}
\newtheorem{prop}[thm]{Proposition}
\newtheorem{conj}[thm]{Conjecture}
\newtheorem{lem}[thm]{Lemma}
\newtheorem{cor}[thm]{Corollary}
\newtheorem{axiom}[thm]{Axiom}
\newtheorem{sheep}[thm]{Corollary}
\newtheorem{deff}[thm]{Definition}
\newtheorem{fact}[thm]{Fact}
\newtheorem{example}[thm]{Example}
\newtheorem{slogan}[thm]{Slogan}
\newtheorem{remark}[thm]{Remark}
\newtheorem{quest}[thm]{Question}
\newtheorem{zample}[thm]{Example}

\newcommand{\sthat}{\hspace{.1cm}| \hspace{.1cm}}
\newcommand{\id}{\operatorname{id} }
\newcommand{\acl}{\operatorname{acl}}
\newcommand{\dcl}{\operatorname{dcl}}
\newcommand{\irr}{\operatorname{irr}}
\newcommand{\aut}{\operatorname{Aut}}
\newcommand{\fix}{\operatorname{Fix}}

\newcommand{\oo}{\mathcal{O}}
\newcommand{\aaa}{\mathcal{A}}
\newcommand{\mm}{\mathcal{M}}
\newcommand{\curg}{\mathcal{G}}
\newcommand{\bbf}{\mathbb{F}}
\newcommand{\A}{\mathbb{A}}
\newcommand{\R}{\mathbb{R}}
\newcommand{\Q}{\mathbb{Q}}
\newcommand{\C}{\mathbb{C}}
\newcommand{\cc}{\mathcal{C}}
\newcommand{\dd}{\mathcal{D}}
\newcommand{\N}{\mathbb{N}}
\newcommand{\Z}{\mathbb{Z}}
\newcommand{\cF}{\mathcal F}
\newcommand{\cB}{\mathcal B}
\newcommand{\cU}{\mathcal U}
\newcommand{\cV}{\mathcal V}
\newcommand{\cG}{\mathcal G}
\newcommand{\cD}{\mathcal D}
\newcommand{\curly}{\mathcal{C}}
\newcommand{\durly}{\mathcal{D}}
\newcommand{\fff}{\mathcal{F}}
\newcommand{\calc}{\mathcal{C}}
\newcommand{\GG}{\mathbb{G}}
\newcommand{\PP}{\mathbb{P}}
\newcommand{\Gal}{\mathrm{Gal}}
\newcommand{\Aut}{\mathrm{Aut}}
\newcommand{\signature}{\mathrm{sign}}

\definecolor{mypink3}{cmyk}{0, 0.7808, 0.4429, 0.1412}

\newcommand{\mahrad}[1]{{\color{blue} \sf $\clubsuit\clubsuit\clubsuit$ Mahrad: [#1]}}
\newcommand{\ramin}[1]{{\color{red}\sf $\clubsuit\clubsuit\clubsuit$ Ramin: [#1]}}
\newcommand{\ramintak}[1]{{\color{mypink3}\sf $\clubsuit\clubsuit\clubsuit$ Ramin2: [#1]}}
\newcommand{\ramintakloo}[1]{{\color{green} \sf $\clubsuit\clubsuit\clubsuit$ Ramin3 [#1]}}

\DeclareRobustCommand{\hlgreen}[1]{{\sethlcolor{green}\hl{#1}}}
\DeclareRobustCommand{\hlcyan}[1]{{\sethlcolor{cyan}\hl{#1}}}
\newcommand{\Fmodtor}{F^\times / \mu(F)}

\author{Sarthak Chimni} 
\address{Department of Mathematics, Statistics, and Computer Science, University of Illinois at Chicago, 851 S Morgan St (M/C 249), Chicago, IL 60607}
\email{sarthakchimni@gmail.com}
\author{Ramin Takloo-Bighash}
\email{rtakloo@uic.edu}

\begin{abstract} In this note we study the distribution of the subrings of $\Z[t]/(t^4)$ and prove two results. The first result gives an asymptotic formula for the number of 
subrings of $\Z[t]/(t^4)$ of bounded index. The method of proof of this theorem is $p$-adic integration a la Grunewald, Segal, and Smith \cite{GSS}. Our second result is about 
the distribution of cocyclic subrings in $\Z[t]/(t^4)$. Our proof of this result is combinatorial and is based on counting certain classes of matrices with Hermite normal forms of a special form.   \end{abstract}

\maketitle

\section{Introduction}  
Given a ring $R$ whose additive group is isomorphic to $\Z^n$ for some $ n \in \mathbb{N}$ we define
\begin{equation*}
a_R^<(k) := |\{ S \text{ subring of } R \mid [R:S] =k\}|
\end{equation*}
For any $ k \in \mathbb{N}$, $a_R^<(k)$ is finite. We define the \textit{subring zeta function of R} by 
\begin{equation*}
\zeta_R^<(s) := \sum_{k=1}^{\infty}\frac{a_R^<(k)}{k^s} = \sum_{S \leq R}\frac{1}{{[R:S]}^s}
\end{equation*} 
One can study the distribution of subrings of finite index in $R$ by using the analytic properties of $\zeta_R^<(s)$ as a function of the complex variable $s$. In particular, by various Tauberian theorems, the location of poles and their orders give information about the function defined by
\begin{equation*}
s_R^<(B) = \sum_{k \leq B} a_R^<(k) = |\{S \text{ subring of } R \mid [R:S] \leq B\}|. 
\end{equation*}

\

Several papers \cite{Brakenhoff, KMT, Liu, N} have considered the problem of determining the analytic properties of the zeta function $\zeta_R^<$ and the growth function $s_R^<(B)$ for rings $R$ of small rank. The rings considered in these papers, $\Z^n$ or rings of integers of number fields, have all been reduced. The ring $\Z[t]/(t^4)$ considered in this note is not reduced. 

\

We prove the following theorem in \S \ref{subring-zeta-1}. 

\begin{thm}\label{thm-subring-zeta} Let $R = \Z[t]/(t^4)$. Then $\zeta_R^<(s) = \prod_p \zeta_{R, p}(s)$ with the product over all primes $p$ of $\Q$, and where for $p$ odd 
\[
\zeta_{R,p}(s)  = \frac{1+px^2+p^2x^3-p^4x^5-2p^5x^6-2p^6x^7-p^7x^8+p^9x^{10}+ p^{10}x^{11}+p^{11}x^{13} }{(1-p^5x^6)(1-p^4x^4)(1-p^3x^3)(1-p^2x^2)(1-x)}
\]
with $x = p^{-s}$.  The right most pole of $\zeta_R^<(s)$ is a simple pole at $s=3/2$. We have 
$$
s_R^<(B) \sim C B^\frac{3}{2}, \; \; B \to \infty
$$
with $C$ equal to 
$\left( \frac{29256 + 18556\sqrt{2}}{194481}\right) \zeta(\frac{3}{2})^2 \zeta(2)\zeta(4)$ multiplied by 
$$
\prod_{p \ne 2} (1 + p^{-2} + p^{-\frac{5}{2}}- p^{-\frac{7}{2}}- 2p^{-4}-2p^{-\frac{9}{2}}-p^{-5}+p^{-6}+p^{-\frac{13}{2}} + p^{-\frac{17}{2}}). 
$$
\end{thm}
We note that for every example of a reduced rank four ring $R$ known to us the growth of $s_R^<(B)$ is of the form $C B(\log B)^m$ for some non-negative integer $m$.

\

In \S \ref{cocyclic-rings} we count the cocyclic subrings of $\Z[t]/(t^4)$. In \cite{NS}, Nguyen and Shparlinski count cocylic sublattices of $\Z^n$. Chinta, Kaplan and Koplewitz \cite{CKK} extend these results to sublattices of corank at most $m$ for $m \leq n$.   In our work \cite{CT} we adapted these ideas to subrings in $\Z[t]/(t^3)$. Here we consider the case of $\Z[t]/(t^4)$. 
We call a subring $S$ of finite index in $R=\Z[t]/(t^4)$ {\em cocyclic} if the finite abelian group $R/S$ is cyclic.  Given a natural number $k$, we let $a_R^{cc}(k)$ be the number of cocyclic subrings of $R$ which are of index $k$ in $R$. 
We also let 
$$
\zeta_R^{cc}(s)= \sum_{k=1}^\infty \frac{a_R^{cc}(k)}{k^s}. 
$$
It is known that 
$$
\zeta_R^{cc}(s) = \prod_p \zeta_{R, p}^{cc}(s)
$$
where each local factor $\zeta_{R, p}^{cc}(s)$ is a rational function in $p^{-s}$.  We prove the following theorem: 
\begin{thm}\label{thm:cocyclic}
For $p$ odd, 
$$
\zeta_{R, p}^{cc}(s) = \frac{1 +x +(p^3-p^2)x^3 -p^3x^4 -p^4x^5 }{(1-p^2x^2)(1-p^4x^4)}. 
$$
The zeta function $\zeta_R^{cc}(s)$ has a simple pole at $s=3/2$. We have 
$$
\sum_{ {k \leq B \atop k {\text{ odd}}}} a_R^{cc}(k) \sim D B^{\frac{3}{2}}, \; \; B \to \infty 
$$
with 
$$
D = \zeta(2) (A + B \sqrt{2}) \prod_{p \ne 2} (1+ 2 p^{-\frac{3}{2}}-p^{-\frac{5}{2}}-p^{-3}-p^{-\frac{7}{2}}).
$$
for explicitly computable rational numbers $A, B$ (see the end of \S \ref{cocyclic-rings} for the exact value). 
\end{thm}
In contrast with our previous work \cite{CT} where we used $p$-adic integration methods to determine 
the full co-type zeta function, here we use combinatorial methods and count certain matrices in Hermite normal form that give cocyclic subrings. 

\

The second author wishes to thank the Simons Foundation for partial support of his work through a Collaboration Grant. The authors also wish to thank Gautam Chinta and Nathan Kaplan for helpful conversations. Some of the numerical computations of this note were performed using {\tt sagemath}.  As stated in an early version of \cite{AKKM},  for $p$ odd the local computations  of Theorem \ref{thm-subring-zeta} were performed before us by Christopher Voll who used the {\tt Zeta} package \cite{Rossmann} developed by Tobias Rossmann. 

\

The paper is organized as follows: In \S \ref{subring-zeta-1} we review the basic methodology of \cite{GSS}, set up our $p$-adic integral, and prove Theorem \ref{thm-subring-zeta}.  We present the proof of Theorem \ref{thm:cocyclic} in \S \ref{cocyclic-rings}. 


\section{Subring zeta function of $\Z[t]/(t^4)$}\label{subring-zeta-1}

The following theorem is a summary of results from \cite{GSS}: 
\begin{thm} \label{GSS}
1. \space The series $\zeta_R^<(s)$ converges in some right half plane of $\mathbb{C}$. The abscissa of convergence $\alpha_R^< $ of $\zeta_R^<(s)$ is a rational number. There is a $\delta > 0$ such that $\zeta_R^<(s)$ can be meromorphically continued to the domain $\{s \in \mathbb{C} \mid \mathcal{R}(s) > \alpha_R^< - \delta\}$. Furthermore, the line $\mathcal{R}(s) = \alpha_R^<$ contains at most one pole of $\zeta_R^<(s)$ at the point $ s = \alpha_R^<$.\\\\
2. \space There is an Euler product decomposition 
\begin{equation}
\zeta_R^<(s) = \prod_p \zeta_{R,p}^<(s)
\end{equation}
with the local Euler factor given by 
\begin{equation*}
\zeta_{R,p}^<(s) = \sum_{l=0}^{\infty} \frac{a_R^<(p^l)}{p^{ls}}
\end{equation*}
This local factor is a rational function of $p^{-s}$; there are polynomials $P_p, Q_p \in \Z[x]$ such that $\zeta_R^<(s) = P_p(p^{-s})/Q_p(p^{-s})$. The polynomials $P_p, Q_p$ can be chosen to have bounded defree as p varies. 
\end{thm}

By a theorem of Voll \cite{Voll}, the local Euler factors satisfy functional equations. The paper \cite{GSS} introduced a $p$-adic formalism to study the local Euler factors $\zeta_R^<(s)$. Fix a $\Z$-basis for $R$ and identify $R $ with $\Z^n$. The multiplication in $R$ is given by a bi-additive map
\[
\beta : \Z^n \times \Z^n \to \Z^n
\]
which extends to a bi-additive map
\[
\beta_p : \Z_p^n \times \Z_p^n \to \Z_p^n
\]
giving $R_p = R \otimes_{\Z} \Z_p$ the structure of a $\Z_p$ -algebra. 

\

Let $\mathcal{M}_p(\beta)$ be the subset of the set of $n \times n$ lower triangular matrices $M$ with entries in $\Z_p$ such that if the rows of $M = (x_{ij})$ are denoted by $v_1, \dots v_n$, then for all $i,j$ satisfying $1 \leq i,j \leq n$, there are $p$-adic integers $c_{ij}^1, \dots c_{ij}^n$ such that 
\begin{equation} \label{multiplicativity}
v_i.v_j = \sum_{k=1}^n c_{ij}^kv_k.
\end{equation}
Let $dM$ be the normalized additive Haar measure on $T_n(\Z_p)$, the set of $ n \times n$ lower triangular matrices with entries in $\Z_p$. Proposition 3.1 of \cite{GSS} says :
\begin{equation} \label{p-adic integral}
\zeta^<_{R,p}(s) = (1-p^{-1})^{1-n} \int_{\mathcal{M}_p(\beta)} |x_{11}|^{s-n+1} |x_{22}|^{s-n+2}\cdots |x_{n-1,n-1}|^{s-1}dM
\end{equation}

\

We now apply these considerations to the case of $\Z[t]/(t^4)$. From now on $R = \Z[t]/(t^4)$. Note that the additive structure of $R$ is the same as that of $\Z^4$ and that $\beta = \{1,t,t^2,t^3\}$ is a basis for $R$ as a lattice. Let $\zeta_R(s)$ denote the subring zeta function for $R$, then by Theorem \ref{GSS} there exists an Euler product decomposition  
\[
\zeta_{\Z[t]/(t^4)}(s) = \sum_{n=1}^{\infty} \frac{a_{\Z[t]/(t^4)}(n)}{n^s} =  \prod_{p \text{ prime}} \zeta_{\Z[t]/(t^4),p}(s)
\]
where 
\[
\zeta_{\Z[t]/(t^4),p}(s) = \sum_{k=1}^{\infty} \frac{a_{k,p}}{p^{ks}}.
\]
Here $a_{k,p}$ counts the number of subrings of $\Z[t]/(t^4)$ of index $p^k$. 
Let $\mathcal{M}_4(p)$ be the subset of the set of $4 \times 4$ lower triangular matrices $M$ with entries in $\Z_p$ such that if the rows of $M = (a_{ij})$ are denoted by $v_1, \dots v_4$, then for all $i,j$ satisfying $1 \leq i,j \leq 4$, there are $p$-adic integers $c_{ij}^1, \dots c_{ij}^4$ such that 
\begin{equation} \label{multiplicativity 4}
v_i.v_j = \sum_{k=1}^4 c_{ij}^kv_k.
\end{equation}
The product $v_i.v_j$ is defined as the row vector representing the product of the corresponding polynomials in $R$.
Equation (\ref{p-adic integral}) applied to $\Z[t]/(t^4)$ gives :
\begin{equation} \label{p-adic integral 4}
\zeta_{\Z[t]/(t^4),p}(s) = (1-p^{-1})^{-3} \int_{\mathcal{M}_4(p)} |a_{11}|^{s-3} |a_{22}|^{s-2} |a_{33}|^{s-1}dM.
\end{equation}
\begin{deff}
If $(k,l,r)$ is a $3$-tuple of non-negative integers, we set
\[
\mathcal{M}_4(p;k,l,r) = \Bigg\{ M = 
\begin{bmatrix}
p^k     &   0 &   0   &    0   &     \\
a_{21}     &     p^l & 0 &    0   &   \\
a_{31}  &   a_{32} & p^r  &   0  &   \\
0    &   0  &    0     &   1
\end{bmatrix}
\in \mathcal{M}_4(p) \Bigg\}.
\]
\end{deff}
Let $\mu_p(k;l;r)$ be the volume of $\mathcal{M}_4(p;k,l,r)$ as a subset of $\Z_p^6$. It follows from equation (\ref{p-adic integral 4}) that 
\begin{equation} \label{sum}
\zeta_{\Z[t]/(t^4),p}(s) = \sum_{k,l,r} \frac{p^{2k+l}}{p^{(k+l+r)s}} \mu_p(k,l,r).
\end{equation}
We use  (\ref{multiplicativity 4}) to describe  $\mathcal{M}_4(p;k,l,r)$ in terms of some inequalities so that we can compute  $\mu_p(k,l,r)$.
\begin{lem} \label{inequalities}
$\mathcal{M}_4(p;k,l,r)$ is the set of matrices 
\[
M = \begin{bmatrix}
p^k     &   0 &   0   &    0   &    \\
a_{21}     &     p^l & 0 &    0   &    \\
a_{31}  &   a_{32} & p^r  &   0  &    \\ 
0    &   0  &    0      &   1
\end{bmatrix}
\]
where the following inequalities hold 

\end{lem}

\begin{equation} \label{equation 1}
0 \leq l \leq 2r;
\end{equation}
\begin{equation} \label{equation 2}
0 \leq k \leq l + r;\\
\end{equation}
\begin{equation} \label{main equation}
k + l - r \leq  v(2p^l a_{32} - p^r a_{21}).
\end{equation}
 
\begin{proof}
Since the subring $S$ is closed under multiplication, we have that
\[
{(a_{31}t^3 + a_{32}t^2 + p^rt)}^2 = 2p^ra_{32}t^3 + p^{2r}t^2 = c_1 p^kt^3 + c_2(a_{21}t^3 + p^lt^2) + c_3(a_{31}t^3 + a_{32}t^2 + p^rt)
\]
where the $c_i \in \Z_p$. Equating coefficients we have 
\[
c_3 = 0 \text{ } c_2 = p^{2r-l} \text{ and } c_1 = p^{-k}(2p^ra_{32} -  p^{2r-l}a_{21}), 
\]
and this gives us inequality (\ref{equation 1}) and (\ref{main equation}). Now
\[
(a_{31}t^3 + a_{32}t^2 + p^rt)(a_{21}x^3 + p^lt^2) = p^{l+r}t^3.
\]
Then, 
\[
p^{l+r}t^3 =  b_1p^kt^3 + b_2(a_{21}t^3 + p^lt^2) + b_3(a_{31}t^3 + a_{32}t^2 + p^rt).
\]
Equating coefficients we have $b_2=b_3 =0$ and $b_1 = p^{l+r-k}$ which gives us (\ref{equation 2}).
\end{proof}

We now proceed to evaluate $\sum_{k,l,r} \frac{p^{2k+l}}{p^{(k+l+r)s}} \mu_p(k,l,r)$. We first assume $p \ne 2$, and then treat the case of $p=2$ separately. 

\subsection{The case where $p\ne 2$}

We consider two cases $r \geq l $ and $r < l$. \\\\
\textbf{CASE 1 :} $r  \geq l$

\

In this case we can simplify (\ref{main equation}) to $ k-r \leq v(2a_{32} -p^{r-l}a_{31})$. Now if $r \geq k$ this always holds so that in that case $\mu_p(k,l,r) = 1$  and if $ r < k$ this holds exactly when $2a_{32} = p^{r-l}y$ for some $y \in \Z_p$ and $y \equiv a_{31} \text{ mod } p^{k+l-2r}$. So that $\mu_p(k,l,r) = p^{l-r}.p^{2r-l-k} = p^{r-k}$. \\

\textbf{SUBCASE 1:} $r < k$

\

Our inequalities reduce to the following:
\[
r < k \leq l+r \text{ and } r \geq l.
\]

Note that $k$ and $l$ cannot be 0 and $\mu_p(k,l,r) = p^{r-k}$ from the above argument. So we have
\begin{equation*}
\begin{split}
F_1(s) = & \sum_{\substack{ r < k \leq l+r, \\ 1 \leq l \leq r}} \frac{p^{2k+l}}{p^{(k+l+r)s}} \mu_p(k,l,r) \\
   &  = \sum_{\substack{ r < k \leq l+r, \\ 1 \leq l \leq r}} \frac{p^{2k+l}.p^{r-k}}{p^{(k+l+r)s}} \\
   & =\sum_{l=1}^{\infty} (px)^l \sum_{r=l}^{\infty} (px)^r \sum_{k = r+1}^{r+l} (px)^k \\
   & = \sum_{l=1}^{\infty} (px)^l \sum_{r=l}^{\infty} (px)^r\bigg(\frac{(px)^{r+1}(1-(px)^l)}{1-px}\bigg) \\
   &   =\frac{1}{1-px}\sum_{l=1}^{\infty} (px)^l \sum_{r=l}^{\infty} (px)^r\bigg((px)^{r+1}(1-(px)^l)\bigg)\\
  &  =\frac{px}{1-px}\sum_{l=1}^{\infty} (px)^l \sum_{r=l}^{\infty}\bigg((px)^{2r} -(px)^{2r+l}\bigg)\\
  &  =\frac{px}{1-px}\sum_{l=1}^{\infty} (px)^l\bigg(\frac{(px)^{2l}}{1-p^2x^2} - \frac{(px)^{3l}}{1-p^2x^2}\bigg) \\
  &  = \frac{px}{(1-px)(1-p^2x^2)}\sum_{l=1}^{\infty}\bigg((px)^{3l} - (px)^{4l}\bigg) \\
  &  =\frac{px}{(1-px)(1-p^2x^2)}\bigg(\frac{p^3x^3}{1-p^3x^3} - \frac{p^4x^4}{1-p^4x^4}\bigg).
\end{split}
\end{equation*}
Simplifying gives
\begin{equation}
F_1(s) = \frac{p^4x^4}{(1-p^2x^2)(1-p^3x^3)(1-p^4x^4)}. \label{sum 1}
\end{equation}

\textbf{SUBCASE 2} $r \geq k$
\[
F_2(s) = \sum_{\substack{ r \geq k, \\ r \geq l}} \frac{p^{2k+l}}{p^{(k+l+r)s}} \mu_p(k,l,r)
\]
\[
= \sum_{\substack{ r \geq k, \\ r \geq l}} \frac{p^{2k+l}}{p^{(k+l+r)s}}
\]
\[
=\sum_{l=0}^{\infty} (px)^l \sum_{r=l}^{\infty} x^r \sum_{k= 0}^{r} (p^2x)^{k}
\]
\[
=\sum_{l=0}^{\infty} (px)^l \sum_{r=l}^{\infty} x^r\bigg(\frac{(1-(p^2x)^{r+1})}{1-p^2x}\bigg)
\]
\[
=\frac{1}{1-p^2x}\sum_{l=0}^{\infty} (px)^l \sum_{r=l}^{\infty}\bigg(x^r - (p^2x)(p^2x^2)^r\bigg)
\]
\[
=\frac{1}{(1-p^2x)(1-x)}\sum_{l=0}^{\infty}(px^2)^l - \frac{p^2x}{(1-p^2x)(1-p^2x^2)}\sum_{l=0}^{\infty} (p^3x^3)^l
\]
\[
= \frac{1}{(1-p^2x)(1-x)(1-px^2)} - \frac{p^2x}{(1-p^2x)(1-p^2x^2)(1-p^3x^3)}.
\]

\

Simplifying the expression gives 
\begin{equation} \label{sum 2}
F_2(s) =\frac{1-p^3x^4}{(1-x)(1-p^2x^2)(1-p^3x^3)(1-px^2)}.
\end{equation}
\textbf{CASE 2 :} $r  < l$

\

In this case we can simplify (\ref{main equation}) to $ k+l-2r \leq v(2p^{l-r}a_{32} -a_{31})$.

\

\textbf{SUBCASE 1:} $r < k$

\

Now if $ r < k$ this holds exactly when $a_{31} = p^{l-r}y$ for some $y \in \Z_p$ and $y \equiv 2a_{32} \text{ mod } p^{k-r}$. This implies that $\mu_p(k,l,r) = p^{2r-k-l}$.
In this case our inequalities reduce to the following:
\[
r  < l \leq 2r \text{ and } r < k \leq l+r.
\]
Note that $k$, $l$ and  $r$ cannot be $0$. So we have
\[
F_3(s) = \sum_{\substack{r  < l \leq 2r \\ r < k \leq l+r }}  \frac{p^{2k+l}}{p^{(k+l+r)s}} \mu_p(k,l,r)
\]
\[
= \sum_{\substack{r  < l \leq 2r \\ r < k \leq l+r }} \frac{p^{2r+k}}{p^{(k+l+r)s}}
\] 
\[
= \sum_{r = 1}^{\infty} (p^2x)^r \sum_{l=r+1}^{2r} x^l \sum_{k=r+1}^{r+l}(px)^k
\]
\[
\sum_{r = 1}^{\infty} (p^2x)^r \sum_{l=r+1}^{2r} x^l\bigg(\frac{(px)^{r+1}(1-(px)^l)}{1-px}\bigg)
\]
\[
=\frac{px}{1-px}\sum_{r = 1}^{\infty}(p^3x^2)^r\sum_{l=r+1}^{2r}\bigg(x^l - (px^2)^l\bigg)
\]
\[
=\frac{px}{1-px}\sum_{r = 1}^{\infty}(p^3x^2)^r\bigg(\frac{x^{r+1}(1-x^r)}{1-x} - \frac{(px^2)^{r+1}(1-(px^2)^r)}{1-px^2}\bigg)
\]
\[
= \frac{px^2}{(1-px)(1-x)}\sum_{r = 1}^{\infty}\bigg((p^3x^3)^r - (p^3x^4)^r\bigg) - \frac{p^2x^3}{(1-px)(1-px^2)}\sum_{r = 1}^{\infty}\bigg((p^4x^4)^r - (p^5x^6)^r\bigg)
\]
\[
=\frac{px^2}{(1-px)(1-x)}\bigg(\frac{p^3x^3}{1-p^3x^3}- \frac{p^3x^4}{1-p^3x^4}\bigg) - \frac{p^2x^3}{(1-px)(1-px^2)}\bigg(\frac{p^4x^4}{1-p^4x^4} - \frac{p^5x^6}{1-p^5x^6}\bigg)
\]
\[
= \frac{p^4x^5}{(1-px)(1-p^3x^3)(1-p^3x^4)} - \frac{p^6x^7}{(1-px)(1-p^4x^4)(1-p^5x^6)}
\]
\[
= \frac{p^4x^5\big((1-p^4x^4)(1-p^5x^6) - p^2x^2(1-p^3x^3)(1-p^3x^4)\big)}{(1-px)(1-p^3x^3)(1-p^3x^4)(1-p^4x^4)(1-p^5x^6)}.
\]

\

Some easy manipulation gives 

\begin{equation}
F_3(s) = \frac{p^4x^5(1+px-p^4x^4 - p^8x^9)}{(1-p^3x^3)(1-p^3x^4)(1-p^4x^4)(1-p^5x^6)}. \label{sum 3}
\end{equation}

\

\textbf{SUBCASE 2:} $r \geq k$

\

Here we further look at 2 subcases. If $k+l \leq 2r$ then we have that (\ref{main equation}) always holds. In that case $\mu_p(k,l,r) = 1$. On the other hand if $k+l > 2r$ then $\mu_p(k,l,r) = p^{2r-k-l}$.

\

\textbf{SUBCASE 2a :} $k+l > 2r$  

\

Our inequalities are  $k \leq r  < l \leq 2r$ and $k+l \leq 2r$. We have 
\[
F_4(s) = \sum_{\substack{0 \leq k \leq r < l \leq 2r, k+l > 2r}} = \frac{p^{2k+l}}{p^{(k+l+r)s}} \mu_p(k,l,r)
\]
\[
= \sum_{\substack{0 \leq k \leq r < l \leq 2r, k+l > 2r}} \frac{p^{2r+k}}{p^{(k+l+r)s}}
\] 
\[
= \sum_{r=1}^{\infty} (p^2x)^r \sum_{l=r+1}^{2r} x^l \sum_{k=2r-l+1}^{r} (px)^k
\]
\[
=  \sum_{r=1}^{\infty} (p^2x)^r \sum_{l=r+1}^{2r} x^l \bigg( \frac{(px)^{2r-l+1}(1-(px)^{l-r})}{1-px}\bigg)
\]
\[
= \frac{px}{1-px} \sum_{r=1}^{\infty} (p^4x^3)^r \sum_{l=r+1}^{2r}\bigg( p^{-l} - (px)^{-r}x^l \bigg) 
\]
\[
= \frac{px}{1-px} \sum_{r=1}^{\infty} (p^4x^3)^r \bigg( \frac{p^{-r-1}(1-p^{-r})}{1-p^{-1}} - (px)^{-r}\bigg( \frac{x^{r+1}(1-x^r)}{1-x}\bigg)\bigg)
\]

\[
=\frac{x}{(1-px)(1-p^{-1})}\sum_{r =1}^{\infty}\bigg((p^3x^3)^r- (p^2x^3)^r\bigg) - \frac{px^2}{(1-px)(1-x)}\sum_{r =1}^{\infty}\bigg((p^3x^3)^r -(p^3x^4)^r\bigg)
\]
\[
= \frac{x}{(1-px)(1-p^{-1})}\bigg(\frac{p^3x^3}{1-p^3x^3} - \frac{p^2x^3}{1-p^2x^3}\bigg) - \frac{px^2}{(1-px)(1-x)}\bigg(\frac{p^3x^3}{1-p^3x^3} - \frac{p^3x^4}{1-p^3x^4}\bigg)
\]
\[
= \frac{p^3x^4}{(1-px)(1-p^3x^3)(1-p^2x^3)} - \frac{p^4x^5}{(1-px)(1-p^3x^3)(1-p^3x^4)}
\]
\[
= \frac{p^3x^4}{(1-px)(1-p^3x^3)}\bigg(\frac{1}{1-p^2x^3} - \frac{px}{1-p^3x^4}\bigg). 
\]

We obtain 

\begin{equation}
F_4(s) = \frac{p^3x^4}{(1-p^3x^3)(1-p^3x^4)(1-p^2x^3)} \label{sum 4}.
\end{equation}

\

\textbf{SUBCASE 2b :} $k+l \leq 2r$ 

\

We have 

\[
F_5(s) = \sum_{\substack{0 \leq k \leq r < l \leq 2r \\ k+l  \leq 2r}} \frac{p^{2k+l}}{p^{(k+l+r)s}} \mu_p(k,l,r)
\]
\[
= \sum_{\substack{0 \leq k \leq r < l \leq 2r \\ k+l  \leq 2r}} \frac{p^{2k+l}}{p^{(k+l+r)s}}
\]
\[
= \sum_{r=1}^{\infty}x^r \sum_{r+1}^{2r}(px)^l\sum_{k =0}^{2r-l}(p^2x)^k
\]
\[
= \sum_{r=1}^{\infty}x^r \sum_{r+1}^{2r}(px)^l \bigg(\frac{1-(p^2x)^{2r-l+1}}{1-p^2x}\bigg)
\]
\[
= \frac{1}{1-p^2x} \sum_{r=1}^{\infty}x^r \sum_{r+1}^{2r}\bigg((px)^l - (p^2x)^{2r+1}p^{-l}\bigg)
\]
\[
=\frac{1}{1-p^2x} \sum_{r=1}^{\infty}x^r\bigg(\frac{(px)^{r+1}(1-(px)^r)}{1-px} - (p^2x)^{2r+1}\bigg(\frac{p^{-r-1}(1-p^{-r})}{1-p^{-1}}\bigg)\bigg)
\]
\[
\frac{px}{(1-p^2x)(1-px)} \sum_{r=1}^{\infty}\bigg((px^2)^r - (p^2x^3)^r\bigg) - \frac{px}{(1-p^2x)(1-p^{-1})}\sum_{r=1}^{\infty} \bigg((p^3x^3)^r - (p^2x^3)^r\bigg)
\]
\[
= \frac{px}{(1-p^2x)(1-px)}\bigg(\frac{px^2}{1-px^2} - \frac{p^2x^3}{1-p^2x^3}\bigg) - \frac{px}{(1-p^2x)(1-p^{-1})}\bigg( \frac{p^3x^3}{1-p^3x^3} - \frac{p^2x^3}{1-p^2x^3}\bigg)
\]
\[
= \frac{p^2x^3}{(1-p^2x)(1-px^2)(1-p^2x^3)} - \frac{p^4x^4}{(1-p^2x)(1-p^3x^3)(1-p^2x^3)}
\]
\[
= \frac{p^2x^3}{(1-p^2x)(1-p^2x^3)}\bigg(\frac{1}{1-px^2} - \frac{p^2x}{1-p^3x^3}\bigg). 
\]

We obtain 

\begin{equation}
F_5(s)= \frac{p^2x^3}{(1-px^2)(1-p^2x^3)(1-p^3x^3)}. \label{sum 5}
\end{equation}

\

Summing up equations (\ref{sum 1})-(\ref{sum 5}) gives 
\[
\zeta_{R,p}(s)  = \frac{1+px^2+p^2x^3-p^4x^5-2p^5x^6-2p^6x^7-p^7x^8+p^9x^{10}+ p^{10}x^{11}+p^{11}x^{13} }{(1-p^5x^6)(1-p^4x^4)(1-p^3x^3)(1-p^2x^2)(1-x)}
\]
with $x = p^{-s}$.  It is now easy to see that 
$$
\prod_{p \ne 2} \zeta_{R, p}(s)
$$
has a simple pole at $s =3/2$ and otherwise it is holomorphic on a domain containing $\Re s \geq 3/2$. In order to prove the theorem we need to show that $\zeta_{R, 2}(s)$ is holomorphic on a domain containing $\Re s \geq 3/2$. We do this next.

\subsection{The case where $p=2$} Here too we consider two cases $r \geq l $ and $r < l$.

\

\textbf{CASE 1 :} $r  \geq l$

\

In this case we can simplify (\ref{main equation}) to $ k-r \leq v(2a_{32} -2^{r-l}a_{31})$. Now if $r \geq k$ this always holds so that in that case $\mu_2(k,l,r) = 1$. if $ r < k$ there are two further subcases. If $r > l$ then this holds exactly when $a_{32} = 2^{r-l-1}y$ for some $y \in \Z_2$ such that $y \equiv a_{31} \text{ mod } 2^{k+l-2r}$. So that $\mu_2(k,l,r) = 2^{l-r-1}.2^{2r-l-k} = 2^{r-k+1}$ when $r > l$. If $r =l$  then the congruence reduces to $k-r \leq v(2a_{32} - a_{31})$ which holds on a volume of  $\mu_2(k,l,r) = 2^{r-k}$. 

\

\textbf{SUBCASE 1:} $r < k$

\

\textbf{SUBCASE 1a :} $ r > l$

\

Our inequalities reduce to the following:
\[
r < k \leq l+r \text{ and } r > l.
\]

Note that $k$ and $l$ cannot be 0 and $\mu_2(k,l,r) = p^{r-k}$ from the above argument. So we have
\begin{equation*}
\begin{split}
T_1(s) = & \sum_{\substack{ r < k \leq l+r, \\ 1 \leq l < r}} \frac{p^{2k+l}}{p^{(k+l+r)s}} \mu_2(k,l,r) \\
   &  = \sum_{\substack{ r < k \leq l+r, \\ 1 \leq l < r}} \frac{p^{2k+l}.p^{r-k+1}}{p^{(k+l+r)s}} \\
   & =p\sum_{l=1}^{\infty} (px)^l \sum_{r=l+1}^{\infty} (px)^r \sum_{k = r+1}^{r+l} (px)^k \\
   & = p\sum_{l=1}^{\infty} (px)^l \sum_{r=l+1}^{\infty} (px)^r\bigg(\frac{(px)^{r+1}(1-(px)^l)}{1-px}\bigg) \\
   &   =\frac{p}{1-px}\sum_{l=1}^{\infty} (px)^l \sum_{r=l+1}^{\infty} (px)^r\bigg((px)^{r+1}(1-(px)^l)\bigg)\\
  &  =\frac{p^2x}{1-px}\sum_{l=1}^{\infty} (px)^l \sum_{r=l+1}^{\infty}\bigg((px)^{2r} -(px)^{2r+l}\bigg)\\
  &  =\frac{p^2x}{1-px}\sum_{l=1}^{\infty} (px)^l\bigg(\frac{(px)^{2l+2}}{1-p^2x^2} - \frac{(px)^{3l+2}}{1-p^2x^2}\bigg) \\
  &  = \frac{p^2x}{(1-px)(1-p^2x^2)}\sum_{l=1}^{\infty}\bigg((px)^{3l+2} - (px)^{4l+2}\bigg) \\
  &  =\frac{p^2x}{(1-px)(1-p^2x^2)}\bigg(\frac{p^5x^5}{1-p^3x^3} - \frac{p^6x^6}{1-p^4x^4}\bigg).
\end{split}
\end{equation*}
Simplifying gives
\begin{equation}
T_1(s) = \frac{p^7x^6}{(1-p^2x^2)(1-p^3x^3)(1-p^4x^4)}. 
\end{equation}

\textbf{SUBCASE 1b :} $ r = l$

\

Now if $r = l$, then $ l < k \leq 2l$ and $\mu_2(k,l,r) = p^{r-k}$. So that our sum is 
\[
T_2(s) = \sum_{l=1}^{\infty} (px)^{2l} \sum_{k=l+1}^{2l}(px)^k
\]
\[
= \frac{1}{1-px}\sum_{l=1}^{\infty} (px)^{2l}(px)^{l+1}(1-(px)^l)
\]
\[
= \frac{px}{1-px} \sum_{l=1}^{\infty} \big((px)^{3l} - (px)^{4l}\big)
\]
\[
= \frac{px}{1-px}\bigg(\frac{p^3x^3}{1-p^3x^3} - \frac{p^4x^4}{1-p^4x^4}\bigg)
\]
\begin{equation}
T_2(s) = \frac{p^4x^4}{(1-p^3x^3)(1-p^4x^4)}
\end{equation}

\

\textbf{SUBCASE 2} $r \geq k$

\

We need to find 

\[
T_3(s) = \sum_{\substack{ r \geq k, \\ r \geq l}} \frac{p^{2k+l}}{p^{(k+l+r)s}} \mu_p(k,l,r)
\]

Here  $\mu_2(k,l,r) = 1$ and the computation is exactly the same as when $ p \neq 2$. So we have,

\begin{equation} 
T_3 (s) = F_2(s) = \frac{1-p^3x^4}{(1-x)(1-p^2x^2)(1-p^3x^3)(1-px^2)}.
\end{equation}
 
 \

\textbf{CASE 2 :} $r  < l$

\

In this case we can simplify (\ref{main equation}) to $ k+l-2r \leq v(2^{l-r+1}a_{32} -a_{31})$.

\

\textbf{SUBCASE 1:} $r < k$

\

Now if $ r < k$ this holds exactly when $a_{31} = 2^{l-r+1}y$ for some $y \in \Z_2$ such that $y \equiv a_{32} \text{ mod } 2^{k-r-1}$. This implies that $\mu_2(k,l,r) = 2^{2r-k-l}$.
In this case our inequalities reduce to the following:
\[
r  < l \leq 2r \text{ and } r < k \leq l+r.
\]
Note that $k$, $l$ and  $r$ cannot be $0$. In this case our sum is 
\[
T_4(s) = \sum_{\substack{r  < l \leq 2r \\ r < k \leq l+r }}  \frac{p^{2k+l}}{p^{(k+l+r)s}} \mu_p(k,l,r)
\]

Again the computation is identical to the one where $ p \neq 2$. So that
\begin{equation}
T_4(s)= F_3(s) =  \frac{p^4x^5(1+px-p^4x^4 - p^8x^9)}{(1-p^3x^3)(1-p^3x^4)(1-p^4x^4)(1-p^5x^6)}. 
\end{equation}

\

\textbf{SUBCASE 2:} $r \geq k$

\

Here we further look at 2 subcases. If $k+l \leq 2r$ then we have that (\ref{main equation}) always holds. In that case $\mu_2(k,l,r) = 1$. On the other hand if $k+l > 2r$ then $\mu_2(k,l,r) = 2^{2r-k-l}$. In either case the computations are the same as when $ p \neq 2$.

\

\textbf{SUBCASE 2a :} $k+l > 2r$  

\

Our inequalities are  $k \leq r  < l \leq 2r$ and $k+l \leq 2r$. Our sum here is 
\[
T_5(s) = \sum_{\substack{0 \leq k \leq r < l \leq 2r, k+l > 2r}} = \frac{p^{2k+l}}{p^{(k+l+r)s}} \mu_p(k,l,r)
\]
So that
\begin{equation}
T_5(s) = F_4(s) = \frac{p^3x^4}{(1-p^3x^3)(1-p^3x^4)(1-p^2x^3)}. 
\end{equation}

\

\textbf{SUBCASE 2b :} $k+l \leq 2r$ 

\

Here $\mu_2(k,l,r) = 1$.  We have 

\[
T_6(s) = \sum_{\substack{0 \leq k \leq r < l \leq 2r \\ k+l  \leq 2r}} \frac{p^{2k+l}}{p^{(k+l+r)s}} \mu_2(k,l,r)
\]
And again
\begin{equation}
T_6(s)= F_5(s) = \frac{p^2x^3}{(1-px^2)(1-p^2x^3)(1-p^3x^3)}. 
\end{equation}

\

Summing up the contributions from the $T_i(s)$ we get

\begin{equation}
\zeta_{R,2}(s) = \frac{\splitfrac{\bigg(2^{12}x^{13}-2^{12}x^{12} + 2^{11}x^{12} + 2^{10}x^{11} + 2^9x^{10} - 2^7x^8 - 2^7x^7}{ + 2^7x^6 - 2^6x^7 - 2^6x^6 - 2^5x^6 - 2^4x^5 + 2^2x^3 + 2x^2 + 1\bigg)}}{(1-x)(1-2x^2)(1-2^4x^4)(1-2^3x^3)(1-2^5x^6)}. 
\end{equation}

\

It is now obvious that $\zeta_{R. 2}(s)$ is holomorphic on a domain containing $\Re s \geq 3/2$, and one can compute its value at $s = 3/2$. 
Now that the local factors of the zeta function are known the rest of the statements of Theorem \ref{thm-subring-zeta} are easy consequences of a standard Tauberian theorem, and basic properties of the Riemann zeta function, plus some rather tedious computations. 

\section{Coyclic Subrings}\label{cocyclic-rings}
In this section we prove Theorem \ref{thm:cocyclic}. 
We start with some definitions.  
\begin{deff}
The \textit{cotype} of a sublattice  $ \Lambda \in \Z^n$ is defined as follows. 
By elementary divisor theory, there is a unique $n$-tuple of integers $(\alpha_1, \dots, \alpha_n) = (\alpha_1(\Lambda), \dots, \alpha_n(\Lambda))
$ such that the finite abelian group $\Z^n/\Lambda$ is isomorphic to the sum of cyclic groups. 
\[
(\Z/\alpha_1\Z) \oplus (\Z/\alpha_2\Z) \oplus \cdots \oplus (\Z/\alpha_n\Z)
\]
where $\alpha_{i+1}|\alpha_i$ for $1 \leq i \leq n-1$. The $n$-tuple $(\alpha_1(\Lambda), \dots, \alpha_n(\Lambda))$ is called the \textit{cotype} of $\Lambda$. The largest index $i$ for which $\alpha_i \neq 1$ is called the \textit{corank} of $\Lambda$. A sublattice $\Lambda$ of corank $0$ or $1$ is called cocyclic, i.e., when $\Z^n/\Lambda$ is cyclic abelian group. 
\end{deff}
From now on set $R = \Z[t]/(t^4)$. A subring $S$ of $R$ is said to be cocyclic if it is cocyclic as an additive  sublattice of $R$. Define $a_R^{cc}(k)$ to be the number of cocyclic subrings of $R$ of index $k$ then the subring zeta function that counts these is given by
\[
\zeta_{\Z[t]/(t^4)}^{cc}(s)  = \sum_{n=1}^{\infty} \frac{a_R^{cc}(k)}{k^s}
\]
It follows from lemma 1.1 in \cite{Petrogadsky} that there is an Euler Product for $\zeta_{\Z[t]/(t^4)}^{cc}(s)$ given by 
\begin{equation*}
\zeta_{\Z[t]/(t^4)}^{cc}(s) = \prod_p \zeta_{p,k}^{cc}(s)
\end{equation*}
where
\[
\zeta_{p,k}^{cc}(s) = \sum_{m=0}^{\infty} \frac{a_{R}^{cc}(p^m)}{p^{ms}}.
\]
The following two propositions allow us to compute $\zeta_{p,k}^{cc}(s)$ by counting a certain class of matrices.  Proposition  \ref{cocyclic} is a consequence of Proposition 8.1 in \cite{MR}, and Propostion \ref{Hermite} is  proposition 2.1 from  \cite{Liu}.

\begin{prop} \label{cocyclic}
If $S$ is a  sublattice of $\Z^n$ of full rank generated by the rows of a matrix $M$ then $R/S \cong (\Z/\alpha_1\Z) \oplus (\Z/\alpha_2\Z) \oplus \cdots \oplus(\Z/\alpha_n\Z)$ where $\alpha_{i+1} \mid \alpha_i$ and $\alpha_{n-k+1}\alpha_{n-k+2} \cdots \alpha_n$ is equal to the gcd of all $k \times k$ minors of $M$, with the convention that if all $k \times k$ minors are $0$, then their gcd is $0$.
\end{prop}

\begin{prop} \label{Hermite}
There is a bijection between lattices $ L \subset \Z^n$ of index $k$ and $n \times n$ lower triangular matrices 
\begin{equation*}
\begin{bmatrix}
a_{11}   &  0   &   \cdots   &  0   \\
a_{21}    & a_{22}    &    \cdots   &    0 \\
\vdots    &   \vdots    &    \ddots     &    0 \\\
a_{n1}   &    a_{n2}    &    \cdots    
\end{bmatrix}
\end{equation*}
with determinant k such that $0 \leq a_{ij} < a_{jj}, for 1 \leq j < i \leq n$. 

\end{prop}
\begin{prop} 
There is a bijection between cocyclic subrings $S$ of  $\Z[t]/(t^4)$ of index $p^{k+l+r}$  and $4 \times 4$ upper triangular matrices $ M \in M_4(\Z)$ 
\[
M = 
\begin{bmatrix}
p^k   &    0 &    0 &    0 \\
a_{21}  &  p^l   & 0 &   0 \\
a_{31}   &   a_{32} &  p^r  & 0 \\ 
0 &   0 &    0 &    1
\end{bmatrix}
\]
such that 
\begin{subequations} 
\begin{equation} \label{Hermite k}
0 \leq a_{21}, a_{31} < k;
\end{equation}
\begin{equation} \label{Hermite l}
0 \leq a_{32} < l;
\end{equation}
\begin{equation}
0 \leq l \leq 2r;
\end{equation}
\begin{equation}
0 \leq k \leq l+r;
\end{equation}
\begin{equation} \label{main}
p^{r-l-k}(2p^la_{32} - p^ra_{21}) \in \Z;
\end{equation}
\begin{equation}\label{3 minor}
\gcd \big(p^{k+l+r}, p^{l+r}, p^ra_{21}, (a_{21}a_{32}-p^la_{31}), p^{k+r}, p^ka_{32}, p^{k+l}\big) =  1.
\end{equation}
\end{subequations}
\end{prop}
\begin{proof}
It follows from proposition \ref{cocyclic} that if $S$ is a subring of $R$ generated by the rows of a matrix $M$ then $S$ is cocyclic if and only if the gcd of the $3 \times 3$ minors of $M$ is $1$, and this is exactly equation (\ref{3 minor}). Equations  (\ref{Hermite k}) and (\ref{Hermite l}) are the conditions from Proposition \ref{Hermite} that give the bijection.  The remaining equations are multiplicativity conditions from Lemma \ref{inequalities}.
\end{proof}

We now proceed to prove Theorem \ref{thm:cocyclic}.  Let $c_{(k,l,r)}$ be the number of matrices that satisfy all the equations  (\ref{Hermite k}) - (\ref{3 minor}) for a fixed choice of $k,l,r$. Then $$a_R^{cc}(p^m) = \sum_{k+l+r = m} c_{(k,l,r)}$$ and 
\[
\zeta_{p,k}^{cc}(s) = \sum_{\substack{l \leq 2r \\ k \leq l+r}}^{\infty} \frac{c_{(k,l,r)}}{p^{(k+l+r)s}}.
\]
As in the previous section we do this in cases.  First we assume $p$ is odd, and then we treat the case of $p=2$ separately.  In each case we find $c_{(k,l,r)}$ and sum over all possible choices of $k,l$ and $r$. In what follows $ x = p^{-s}$.

\subsection{The case where $p\ne 2$}. We proceed by writing down cases. 

\

\textbf{CASE 1 : } $ k=0 $

\

In this case $a_{21}=a_{31} = 0$. If $l=0$ then we have a unique cocyclic subring corresponding to each choice of $r$ so that $c_{(0,0,r)} = 1$ for all $r$. If $l \geq 1$ then equation (\ref{3 minor}) reduces to $ p \nmid a_{32}$, giving  $(p^l - p^{l-1})$ choices for $M$ and $c_{(0,l,r)} = p^l - p^{l-1}$.  So that our sum becomes 
\[
\zeta_1(s) = \sum_{r =0}^{\infty} x^r + \sum_{r=1}^{\infty} x^r \sum_{l=1}^{2r}(p^l-p^{l-1})x^l
\]
\[
=  \frac{1}{1-x} + \frac{(p-1)x}{1-px} \sum_{r=1}^{\infty} x^r(1-(px)^{2r})
\]
\[
=\frac{1}{1-x} + \frac{(p-1)x}{1-px} \bigg(\frac{x}{1-x} - \frac{p^2x^3}{1-p^2x^3}\bigg)
\]
\[
=\frac{1}{1-x} + \frac{(p-1)(1+px)x^2}{(1-x)(1-p^2x^3)}.
\]
Simplification gives 
\begin{equation}
\zeta_1(s) = \frac{px^2+x+1}{1-p^2x^3}. 
\end{equation}

\

\textbf{CASE 2 : }$ k > 0, l = 0$ 

\

Note that we have $a_{32} = 0$ and $ k \leq r$ and all the conditions reduce to $p \nmid a_{31}$.  We have $p^k - p^{k-1}$ choices for $a_{31}$ and $p^k$ choices for $a_{21}$ which gives $c_{(k,0,r)} =  p^{2k} - p^{2k-1}$. So we have 
\[
\zeta_2(s) = \sum_{r=1}^{\infty} x^r \sum_{k=1}^{r} p^k(p^k-p^{k-1})x^k
\]
\[
= \sum_{r=1}^{\infty} x^r\sum_{k=1}^{r}(1-p^{-1})(p^2x)^k
\]
\[
=  \frac{p^2x(1-p^{-1})}{1-p^2x}\sum_{r=1}^{\infty} x^r(1-(p^2x)^r)
\]
\[
= \frac{p(p-1)x}{1-p^2x}\bigg(\frac{x}{1-x} - \frac{p^2x^2}{1-p^2x^2}\bigg).
\]

This gives

\begin{equation}
\zeta_2(s)= \frac{p(p-1)x^2}{(1-x)(1-p^2x^2)}.
\end{equation}

\

\textbf{CASE 3 : } $ k ,l,r \geq 1$ 

\

In this case (\ref{3 minor}) implies that both $a_{21}$ and $a_{32}$  are coprime to $p$.  Therefore (except in 3b) we have that $c_{(k,l,r)} = (p^{2k}-p^{2k-1})(p^l-p^{l-1})$) in all the following subcases: 

\

SUBCASE 1 : $ l < r$  

\

We can simplify (\ref{main}) as $p^{r-k}(2a_{32} - p^{r-l}a_{21}) \in \Z$. Since $ l < r$, this holds only when $k \leq r$. So our sum is 
\[
\zeta_3(s) = \sum_{r=2}^{\infty}x^r \sum_{k=1}^{r} (p^{2k}-p^{2k-1})x^k \sum_{l=1}^{r-1} (p^l - p^{l-1})x^l
\]
\[
= \frac{px(1-p^{-1})^2}{1-px} \sum_{r=2}^{\infty}x^r(1 -(px)^{r-1}) \sum_{k=1}^{r} (p^2x)^k
\]
\[
= \frac{p(p-1)^2x^2}{(1-px)(1-p^2x)}\sum_{r=2}^{\infty} (x^r - (px)^{-1}(px^2)^r)(1-(p^2x)^r)
\]
\[
= \frac{p(p-1)^2x^2}{(1-px)(1-p^2x)}\sum_{r=2}^{\infty}\bigg(x^r - (p^2x^2)^r - (px)^{-1}(px^2)^r + (px)^{-1} (p^3x^3)^r\bigg)
\]
\[
= \frac{p(p-1)^2x^2}{(1-px)(1-p^2x)}\bigg(\frac{x^2}{1-x} - \frac{p^4x^4}{1-p^2x^2} - \frac{px^3}{1-px^2} + \frac{p^5x^5}{1-p^3x^3}\bigg).
\]

We have 

\begin{equation}
\zeta_3(s) = \frac{p(p-1)^2x^4(-p^3x^3 -p^2x^2 + p^2x +1)}{(1-x)(1-p^2x^2)(1-p^3x^3)(1-px^2)}
\end{equation}

\

SUBCASE 2 : $ l > r $ 

\

In this case we can simplify (\ref{main}) to $p^{2r-k-l}(2p^{l-r}a_{32} - a_{21}) \in \Z$. This only holds when $k+l \leq 2r$. Our sum is 
\[
\zeta_4(s) = \sum_{r=2}^{\infty} x^r \sum_{l=r+1}^{2r-1} (p^l - p^{l-1})x^l \sum_{k=1}^{2r-l} (p^{2k} - p^{2k-1})x^k
\]
(Note that if $r = 1$ then $ l = 2$ which implies $ k = 0$ which was considered in a previous case.)
\[
= \frac{p^2x(1-p^{-1})^2}{1-p^2x}\sum_{r=2}^{\infty} x^r \sum_{l=r+1}^{2r-1} (px)^l (1 - p^{4r - 2l}x^{2r -l})
\]
\[
= \frac{(p-1)^2x}{1-p^2x}\sum_{r=2}^{\infty} x^r \sum_{l=r+1}^{2r-1} (px)^l - p^{4r-l}x^{2r}
\]
$$
=  \frac{(p-1)^2x}{(1-p^2x)(1-px)}\bigg(\sum_{r=2}^{\infty} x^r(px)^{r+1}(1-(px)^{r-1})\bigg) 
$$
$$
- \frac{p(p-1)x}{(1-p^2x)}\bigg(\sum_{r=2}^{\infty} x^r(p^4x^2)^r p^{-r-1}(1-p^{-r+1})\bigg)
$$
\[
=\frac{(p-1)^2x}{(1-p^2x)(1-px)}\bigg(\frac{p^3x^5}{1-px^2} - \frac{p^4x^6}{1-p^2x^3}\bigg) - \frac{p(p-1)x}{(1-p^2x)}\bigg(\frac{p^5x^6}{1-p^3x^3} - \frac{p^4x^6}{1-p^2x^3}\bigg).
\]

\

We have 

\begin{equation}
\zeta_4(s) = \frac{p^3(p-1)^2x^6}{(1-px^2)(1-p^3x^3)(1-p^2x^3)}
\end{equation}

\

SUBCASE 3 : $ l = r $ 

\

In this case (\ref{main}) reduces to $ p^{l-k}(2a_{32} -a_{21}) \in \Z$. 

\

SUBCASE 3a) : $k \leq l$ 

\

Equation (\ref{main})  holds for all $a_{32}, a_{21}$ satisfying (\ref{Hermite k}) and (\ref{Hermite l}). Our sum in this case is 
 \[
 \zeta_5(s) = \sum_{l=1}^{\infty} (p^l-p^{l-1})x^{2l} \sum_{k=1}^{l} (p^{2k}-p^{2k-1})x^k
 \]
 \[
 = \frac{p^2(1-p^{-1})^2x}{1-p^2x} \sum_{l=1}^{\infty} (px^2)^l(1-(p^2x)^l)
 \]
 \[
 = \frac{(p-1)^2x}{1-p^2x}\bigg(\frac{px^2}{1-px^2} - \frac{p^3x^3}{1-p^3x^3}\bigg).
 \]
 
 We obtain 
 
\begin{equation}
\zeta_5(s) = \frac{p(p-1)^2x^3}{(1-px^2)(1-p^3x^3)}.
\end{equation}

\

SUBCASE 3b) $ k > l$. 

\

We can rewrite (\ref{main}) as $p^{k-l} \mid 2a_{32} - a_{21}$. Fix $a_{32}$ coprime to $p$ such that $a_{32} < p^l$. There are $p^l - p^{l-1}$ choices for $a_{32}$. Since $a_{21}$ is congruent to $2a_{32}$ modulo $p^{k-l}$ that leaves $p^{k-(k-l)} = p^l$ choices for $a_{32}$. So that $c_{(k,l,r)} =  p^k(p^{2l}-p^{2l-1})$.
\[
\zeta_6(s) = \sum_{l=1}^{\infty} (p^{2l} - p^{2l-1})x^{2l} \sum_{k = l+1}^{2l} (px)^k
\]
\[
= \frac{1-p^{-1}}{1-px}  \sum_{l=1}^{\infty} (p^2x^2)^l(px)^{l+1}(1-(px)^l)
\]
\[
= \frac{1-p^{-1}}{1-px}\bigg(\frac{p^4x^4}{1-p^3x^3} - \frac{p^5x^5}{1-p^4x^4}\bigg). 
\]

We have 

\begin{equation}
\zeta_6(s) = \frac{p^3(p-1)x^4}{(1-p^3x^3)(1-p^4x^4)}. 
\end{equation}

\

Now $\zeta_{R, p}^{cc}(s) = \sum_{j=1}^{6} \zeta_j(s)$. Putting everything together we obtain 
\begin{equation*}
\zeta_{R, p}^{cc}(s) = \frac{-p^3x^4 + (p^2-p)x^2 + (1-p)x+1}{(1-px)(1-p^4x^4)}. 
\end{equation*}

It is easy to see that $\prod_{p \ne 2} \zeta_{R, p}^{cc}(s)$ is meromorphic on a domain containing $\Re s \geq 3/2-\delta$ for some $\delta > 0$ with a single simple pole at $s=3/2$. 
In order to apply the Tauberian theorem we need to show that $\zeta_{R, 2}(s)$ is holomorphic on a domain containing $\Re s \geq 3/2$ and compute its value at $s = 3/2$. We do this next. 

\subsection{The case where $p = 2$} We recognize several cases. 

\

\textbf{CASE 1 : } $ k=0 $

\

In this case $a_{21}=a_{31} = 0$. If $l=0$ then we have a unique cocyclic subring corresponding to each choice of $r$ so that $c_{(0,0,r)} = 1$ for all $r$. If $l \geq 1$ then equation (\ref{3 minor}) reduces to $ 2 \nmid a_{32}$, giving  $(2^l - 2^{l-1})$ choices for $M$ and $c_{(0,l,r)} = 2^l - 2^{l-1}$.  So that our sum becomes 
\[
E_1(s) = \sum_{r =0}^{\infty} x^r + \sum_{r=1}^{\infty} x^r \sum_{l=1}^{2r}(p^l-p^{l-1})x^l
\]
This is the the exact same sum as for $p \neq 2$. So that

\begin{equation}
E_1(s) = \frac{px^2+x+1}{1-p^2x^3}. 
\end{equation}

\

\textbf{CASE 2 : }$ k > 0, l = 0$ 

\

Note that we have $a_{32} = 0$ and $ k \leq r$ and all the conditions reduce to $2 \nmid a_{31}$.  We have $2^k - 2^{k-1}$ choices for $a_{31}$ and $2^k$ choices for $a_{21}$ which gives $c_{(k,0,r)} =  2^{2k} - 2^{2k-1}$. So we have that
\[
E_2(s) = \sum_{r=1}^{\infty} x^r \sum_{k=1}^{r} p^k(p^k-p^{k-1})x^k
\]
Again there is no difference between the odd prime case and the $p=2$ case. Hence,

\begin{equation}
E_2(s)= \frac{p(p-1)x^2}{(1-x)(1-p^2x^2)}.
\end{equation}

\

\textbf{CASE 3 : } $ k ,l,r \geq 1$ 

\

In this case (\ref{3 minor}) implies that both $a_{21}$ and $a_{32}$  are coprime to $2$.  Therefore (except in 3b) we have that $c_{(k,l,r)} = (2^{2k}-2^{2k-1})(2^l-2^{l-1})$) in all the following subcases: 

\

SUBCASE 1 : $ l < r-1$  

\

We can simplify (\ref{main}) as $2^{r-k+1}(a_{32} - 2^{r-l-1}a_{21}) \in \Z$. Since $ l < r-1$, this holds only when $k-1 \leq r$. So our sum is 
\[
E_3(s) = \sum_{r=3}^{\infty}x^r \sum_{k=1}^{r+1} (p^{2k}-p^{2k-1})x^k \sum_{l=1}^{r-2} (p^l - p^{l-1})x^l
\]
\[
= \frac{px(1-p^{-1})^2}{1-px} \sum_{r=3}^{\infty}x^r(1 -(px)^{r-2}) \sum_{k=1}^{r+1} (p^2x)^k
\]
\[
= \frac{p(p-1)^2x^2}{(1-px)(1-p^2x)}\sum_{r=3}^{\infty} (x^r - (px)^{-2}(px^2)^r)(1-(p^2x)^{r+1})
\]
\[
= \frac{p(p-1)^2x^2}{(1-px)(1-p^2x)}\sum_{r=3}^{\infty}\bigg(x^r - (p^2x)(p^2x^2)^r - (px)^{-2}(px^2)^r + x^{-1}(p^3x^3)^r\bigg)
\]
\[
= \frac{p(p-1)^2x^2}{(1-px)(1-p^2x)}\bigg(\frac{x^3}{1-x} - \frac{p^8x^7}{1-p^2x^2} - \frac{px^4}{1-px^2} + \frac{p^9x^8}{1-p^3x^3}\bigg).
\]

We have 

\begin{equation}
E_3(s) = (p(p-1)^2x^5)\left(\frac{\splitfrac{p^7x^6 - p^7x^5 - p^6x^4 + p^5x^5 + p^6x^3}{ - p^5x^4 - p^4x^3 + p^4x^2 - p^3x^3 - p^2x^2 + p^2x + 1}}{(1-x)(1-p^2x^2)(1-p^3x^3)(1-px^2)}\right).
\end{equation}

\

SUBCASE 2 : $ l \geq r $ 

\

In this case we can simplify (\ref{main}) to $2^{2r-k-l}(2^{l-r+1}a_{32} - a_{21}) \in \Z$. This only holds when $k+l \leq 2r$. Our sum is 
\[
E_4(s) = \sum_{r=1}^{\infty} x^r \sum_{l=r}^{2r-1} (p^l - p^{l-1})x^l \sum_{k=1}^{2r-l} (p^{2k} - p^{2k-1})x^k
\]

\[
= \frac{p^2x(1-p^{-1})^2}{1-p^2x}\sum_{r=1}^{\infty} x^r \sum_{l=r}^{2r-1} (px)^l (1 - p^{4r - 2l}x^{2r -l})
\]
\[
= \frac{(p-1)^2x}{1-p^2x}\sum_{r=1}^{\infty} x^r \sum_{l=r}^{2r-1} (px)^l - p^{4r-l}x^{2r}
\]
$$
=  \frac{(p-1)^2x}{(1-p^2x)(1-px)}\bigg(\sum_{r=1}^{\infty} x^r(px)^{r}(1-(px)^r)\bigg) 
$$
$$
- \frac{p(p-1)x}{(1-p^2x)}\bigg(\sum_{r=1}^{\infty} x^r(p^4x^2)^r p^{-r}(1-p^{-r})\bigg)
$$
\[
=\frac{(p-1)^2x}{(1-p^2x)(1-px)}\bigg(\frac{px^2}{1-px^2} - \frac{p^2x^3}{1-p^2x^3}\bigg) - \frac{p(p-1)x}{(1-p^2x)}\bigg(\frac{p^3x^3}{1-p^3x^3} - \frac{p^2x^3}{1-p^2x^3}\bigg).
\]

\

We have 

\begin{equation}
E_4(s) = \frac{p(p-1)^2x^3}{(1-px^2)(1-p^3x^3)(1-p^2x^3)}.
\end{equation}

\

SUBCASE 3 : $ l = r -1 $ 

\

In this case (\ref{main}) reduces to $ 2^{l-k+2}(a_{32} -a_{21}) \in \Z$. 

\

SUBCASE 3a) : $k \leq l+2$ 

\

Equation (\ref{main})  holds for all $a_{32}, a_{21}$ satisfying (\ref{Hermite k}) and (\ref{Hermite l}). Our sum in this case is 
 \[
 E_5(s) = \sum_{l=1}^{\infty} (p^l-p^{l-1})x^{2l+1} \sum_{k=1}^{l+2} (p^{2k}-p^{2k-1})x^k
 \]
 \[
 = \frac{p^2(1-p^{-1})^2x^2}{1-p^2x} \sum_{l=1}^{\infty} (px^2)^l(1-(p^2x)^{l+2})
 \]
 \[
 = \frac{(p-1)^2x^2}{1-p^2x}\bigg(\frac{px^2}{1-px^2} - \frac{p^7x^5}{1-p^3x^3}\bigg).
 \]
 
 We obtain 
 
\begin{equation}
E_5(s) = \frac{p(p-1)^2x^4(1+p^2x-p^3x^3+p^4x^2-p^5x^4)}{(1-px^2)(1-p^3x^3)}.
\end{equation}

\

SUBCASE 3b) $ k > l+2$. 

\

We can rewrite (\ref{main}) as $2^{k-l-2} \mid a_{32} - a_{21}$. Fix $a_{32}$ coprime to $2$ such that $a_{32} < 2^l$. There are $2^l - 2^{l-1}$ choices for $a_{32}$. Since $a_{21}$ is congruent to $a_{32}$ modulo $2^{k-l-2}$ that leaves $2^{k-(k-l-2)} = 2^{l+2}$ choices for $a_{32}$. So that $c_{(k,l,r)} =  2^k(2^{2l+2}-2^{2l+1})$.
\[
E_6(s) = \sum_{l=1}^{\infty} (p^{2l+2} - p^{2l+1})x^{2l+1} \sum_{k = l+3}^{2l} (px)^k
\]
\[
= \frac{(p^2-p)x}{1-px}  \sum_{l=1}^{\infty} (p^2x^2)^l(px)^{l+3}(1-(px)^{l-2})
\]
\[
= \frac{p(p-1)x}{1-px}\bigg(\frac{p^4x^6}{1-p^3x^3} - \frac{p^3x^5}{1-p^4x^4}\bigg). 
\]

We have 

\begin{equation}
E_6(s) = \frac{p^4(p-1)x^6(p^4x^4+p^3x^3-1)}{(1-p^3x^3)(1-p^4x^4)}. 
\end{equation}
\[
\zeta_{R, 2}^{cc}(s) = \sum_{i=1}^6 E_i(s)
\]
\begin{equation}
\zeta_{R, 2}^{cc}(s) = \frac{\left(\splitfrac{1+x+2^2x^2-2^3x^4-2^4x^5+2^4x^6-2^5x^6-2^6x^6}{+2^7x^6+2^7x^8-2^7x^9-2^9x^8+2^8x^9+2^9x^8-2^8x^{10}+2^9x^{10}}\right)}{(1-2^4x^4)(1-2^3x^3)}
\end{equation}
This is easily seen to be holomorphic on a domain containing $\Re s \geq 3/2$. 
A computation shows that the value of $\zeta_{R, 2}^{cc}(3/2)$ is 
$$
\frac{25\sqrt{2}+177}{96 -24\sqrt(2)} \sim 3.422. 
$$
The rest of the statements of Theorem \ref{thm:cocyclic} follow from standard Tauberian theorems and basic properties of the Riemann zeta function.  We note that the quantity $A + B \sqrt{2}$ in the statement of the theorem is equal to $\zeta_{R, 2}^{cc}(3/2)/8$.

\end{document}